\documentclass[final,11pt]{elsarticle}

\usepackage{verbatim,a4wide}

\usepackage{amsmath}
\usepackage{amssymb}
\usepackage{amsthm}

\usepackage{algorithm}
\usepackage{algpseudocode}

\usepackage{subcaption}

\usepackage{xcolor}

\usepackage{tikz}
\usetikzlibrary{patterns}

\usepackage{hyperref}

\usepackage[cp1250]{inputenc}
\usepackage[OT4]{fontenc}
\usepackage[english]{babel}

\usepackage{datetime}

\numberwithin{equation}{section}
\numberwithin{algorithm}{section}
\numberwithin{figure}{section}
\numberwithin{table}{section}

\newtheorem{theorem}{Theorem}[section]
\newtheorem{lemma}[theorem]{Lemma}
\newtheorem{remark}[theorem]{Remark}
\newtheorem{corollary}[theorem]{Corollary}

\newtheorem{assumption}[theorem]{Assumption}

\newcommand{\p}[1]{\mbox{\textsf{#1}}}
\newcommand{\floor}[1]{\left\lfloor{#1}\right\rfloor}
\newcommand{\ceil}[1]{\left\lceil{#1}\right\rceil}

\newcommand{\hyper}[5]{\,{}_{#1}F_{#2}\!\left(%
            \begin{array}{cc}{\displaystyle{#3}}\\[0.25ex]%
            {\displaystyle{#4}} \end{array}\bigg|\,{\displaystyle{#5}}%
            \right)}

\journal{BIT Numerical Mathematics}

\begin{document}

\begin{frontmatter}

\title{Evaluation of Gauss--Legendre curves}

\author[A1]{Filip Chudy}
\ead{fch@cs.uni.wroc.pl}

\author[A1]{Pawe{\l} Wo\'{z}ny\corref{cor}}
\ead{pwo@cs.uni.wroc.pl}

\cortext[cor]{Corresponding author.}

\address[A1]{Institute of Computer Science, University of Wroc{\l}aw,
             ul.~Joliot-Curie 15, 50-383 Wroc{\l}aw, Poland}

\begin{abstract}
We present new representations of Gauss--Legendre polynomials and their 
derivatives in the shifted power basis and in bases related to symmetric 
orthogonal Jacobi polynomials. Using these representations and certain 
recurrence relations, we propose $O(n^2+dn)$ methods for evaluating 
a Gauss--Legendre curve of degree $n$ in $\mathbb E^d$. We also propose 
algorithms for multipoint evaluation with computational complexity 
$O(Mdn+dn^2)$, where $M$ is the number of evaluation points. Numerical tests 
show that the approach using the shifted power basis is very fast, but it works 
well only for small $n$ (approximately up to 20), whereas the method employing 
Jacobi polynomials gives good numerical results even for large $n$ 
(around 100).  
\end{abstract}

%
%
%
%


\end{frontmatter}

\section{Introduction}                                  \label{S:Introduction}

The \textit{Gauss--Legendre curves} (\textit{GL curves} for short), recently 
introduced in~\cite{Moon2023} (see also~\cite{Moon2026}), are polynomial 
parametric curves closely related to Gauss--Legendre quadratures and Legendre 
orthogonal polynomials (see, e.g., \cite[\S3.2]{WG}, \cite[\S9.8.3]{KLS2010}). 
They have recently started gaining traction in CAGD, partly due to their very 
good shape control properties. Unlike B\'{e}zier curves (see, e.g., 
\cite[\S11.4]{Agoston2005}, \cite[\S4--\S6]{Farin2002}), GL curves closely 
follow the control polygon, even at high degrees. Therefore, they may have many 
practical applications, for example in numerical analysis and computer 
graphics.

However, the so-called \textit{Gauss--Legendre polynomials} 
(\textit{GL polynomials} for short), which form the basis of GL curves, are 
defined in a rather complicated manner, and their efficient evaluation is 
nontrivial. Note that the evaluation of points on GL curves is not discussed 
in~\cite{Moon2023}.   

The article is organized as follows. In Section~\ref{S:GL-curves}, we briefly 
describe the GL polynomials and GL curves and recall their main properties. 
The new representations of the GL polynomials and their derivatives are 
presented in Section~\ref{S:NewRep}, where we consider expansions with respect 
to the shifted power basis and some orthogonal bases related to symmetric 
Jacobi polynomials (cf.~\S\ref{S:Jacobi}). Then, in Section~\ref{S:Evaluation}, 
using these representations, we propose fast algorithms for evaluating GL 
curves. Their computational complexity is $O(n^2+dn)$ when computing a point on 
a GL curve of degree $n\in\mathbb N$ in $\mathbb E^d$ $(d\geq1)$. Moreover, we 
investigate the problem of multipoint evaluation of GL curves and propose 
algorithms that run in $O(Mdn+dn^2)$ time, where $M$ denotes the number of 
evaluation points. Usually, $M\gg n$ and $d$ is small, so the computational 
complexity is $O(Mdn)$. Results of the conducted numerical experiments are 
presented in~\S\ref{S:Tests}. 

Numerical tests show that the approach using the shifted power basis is the 
fastest. However, it works well only for small $n$ ($\lessapprox 20$), whereas 
the method employing symmetric Jacobi polynomials gives good numerical results 
even for large $n$.

It should be mentioned that the problems of evaluating and rendering GL 
curves were also considered in~\cite{RH2026}. The conversion between GL and 
B\'{e}zier curves was discussed in~\cite{KSM2026}, where the problems of 
subdivision and degree elevation of GL curves were also addressed.

\section{Symmetric Jacobi polynomials}                        \label{S:Jacobi}

We now recall several well-known properties of orthogonal polynomials from the 
\textit{Jacobi family} (of which the Legendre polynomials are a member) that 
will be useful in the study presented in this paper. For details and proofs of 
the results given below, we refer, for example, to 
\cite[\S8.9, \S8.91, \S8.96]{GR2007}, \cite[\S9.8]{KLS2010}, 
\cite[\S3.2, \S3.3]{STW2011}, and the references given therein.

The \textit{Jacobi polynomials} 
$P^{(\alpha,\beta)}_m\in\Pi_m\setminus\Pi_{m-1}$ $(m\in\mathbb N;\,
\alpha,\beta>-1)$ are orthogonal in the sense that 
$$
\int_{-1}^{1}(1-x)^\alpha (1+x)^\beta 
             P^{(\alpha,\beta)}_i(x)
             P^{(\alpha,\beta)}_j(x)\,\mbox{d}x=\delta_{ij}h_i,
$$
where $\delta_{ij}$ denotes the \textit{Kronecker delta}, $h_i>0$, and
$\Pi_\ell$ is the set of polynomials of degree $\leq \ell$ 
($\Pi_{-1}:=\emptyset$).

We use the standard normalization
$$
P_k^{(\alpha,\beta)}(1)=\binom{k+\alpha}{k}\qquad (k=0,1,\ldots)
$$
which uniquely defines the Jacobi polynomials.

In this paper, we focus on the case $\alpha=\beta$ (which corresponds to the 
so-called \textit{Gegenbauer polynomials}). Thus, for simplicity, we set
\begin{equation}\label{E:Def-JacobiSymm}
P^{(\alpha)}_k(x):=P^{(\alpha,\alpha)}_k(x)\qquad (k=0,1,\ldots;\, \alpha>-1).
\end{equation}

The polynomials $P^{(\alpha)}_k$ have the following explicit representation:
\begin{equation}\label{E:P_k-power}
\displaystyle
P^{(\alpha)}_k(x)=\frac{(\alpha+1)_k}{k!}
               \hyper{2}{1}{-k,\,k+2\alpha+1}
                           {\alpha+1}{\frac{1-x}{2}}\qquad (k=0,1,\ldots),
\end{equation}
where ${}_{p}F_{q}$ is the \textit{generalized hypergeometric function} defined 
by
$$
\hyper{p}{q}{\alpha_1,\ldots, \alpha_p}{\beta_1,\ldots,\beta_q}{x}:=
             \sum_{\ell=0}^{\infty}
                      \frac{(\alpha_1)_\ell\ldots(\alpha_p)_\ell}
                           {(\beta_1)_\ell\ldots(\beta_q)_\ell}
                                                     \cdot\frac{x^\ell}{\ell!},
$$
where $p,q\in\mathbb N$, $\alpha_i\in\mathbb C$ $(i=1,2,\ldots,p)$,
$\beta_j\in\mathbb C$ $(j=1,2,\ldots,q)$, $x\in\mathbb C$, and $(a)_\ell$ 
$(a\in\mathbb C;\ \ell\in\mathbb N)$ denotes the \textit{Pochhammer symbol},
$$
(a)_0:=1,\qquad (a)_\ell:=a(a+1)\ldots(a+l-1)\quad (\ell\geq 1).
$$
Notice that if one of the parameters $\alpha_i$ is equal to $-m$ 
$(m\in\mathbb N)$ then the generalized hypergeometric function is a polynomial 
in $x$ of degree at most $m$. See, e.g., \cite[\S9.1, \S9.14]{GR2007}, 
\cite[\S1.4]{KLS2010}.

They satisfy a recurrence relation, namely
\begin{equation}\label{E:JacobiRecRel}
\displaystyle
P_{k}^{(\alpha)}(x)=x\,\xi^{(\alpha)}_1(k)P_{k-1}^{(\alpha)}(x)-
                               \xi^{(\alpha)}_2(k)P_{k-2}^{(\alpha)}(x)
                                                         \qquad (k=2,3,\ldots),
\end{equation}
where
$$
\xi^{(\alpha)}_1(k):=\frac{(k+\alpha)(2k+2\alpha-1)}{k(k+2\alpha)}
                                                                \quad(k\geq1),
\qquad
\xi^{(\alpha)}_2(k):=\frac{(k+\alpha-1)_2}{k(k+2\alpha)}\quad(k\geq2),
$$
and $P_{0}^{(\alpha)}(x)=1$, 
$P_{1}^{(\alpha)}(x)=\xi^{(\alpha)}_1(1)\,x=(\alpha+1)x$. Hence, for given
$x\in\mathbb R$, $m\in\mathbb N$ and $\alpha>-1$, $P^{(\alpha)}_m(x)$ can be 
evaluated in $O(m)$ time.

For $k\in\mathbb N$, we have the following symmetry:
\begin{equation}\label{E:P_k-symm}
P_k^{(\alpha)}(x)=(-1)^kP_k^{(\alpha)}(-x).
\end{equation}
The $i$th derivative of $P_k^{(\alpha)}$ is also a symmetric Jacobi polynomial, 
i.e.,
\begin{equation}\label{E:Jacobi-diff}
\frac{\mbox{d}^i}{\mbox{d}x^i}P_{k}^{(\alpha)}(x)=
                     \frac{(k+2\alpha+1)_i}{2^i}P_{k-i}^{(\alpha+i)}(x),
\end{equation}
where $k,i\in\mathbb N$ and we adopt the convention---used throughout the 
sequel---that $P_\ell^{(\alpha)}\equiv 0$ for $\ell=-1,-2,\ldots$.

Let us consider the polynomial $u_m\in\Pi_m$ $(m\in\mathbb N)$ given in the 
\textit{symmetric Jacobi form},
$$
u_m(x)=\sum_{k=0}^{m}v_k P_k^{(\alpha)}(x),
$$
where the coefficients $v_0,v_1,\ldots,v_m\in\mathbb R$ are known. Recall that, 
for a given $x\in\mathbb R$, the value $u_m(x)$ can be computed in $O(m)$ time 
using the famous \textit{Clenshaw algorithm} (see, e.g., 
\cite[pp.~124--125]{WG}, \cite[\S10.2]{Wimp1984}):
\begin{enumerate}\label{E:Clenshaw}
\itemsep 1ex

\item Set $B_{m+1}=B_{m+2}:=0$.

\item For $k=m,m-1,\ldots,0$, set
      $$
      B_k:=v_k+x\,\xi^{(\alpha)}_1(k+1)B_{k+1}-\xi^{(\alpha)}_2(k+2)B_{k+2}.
      $$

\end{enumerate}
Then $u_m(x)=B_0$.

Note that the Clenshaw algorithm follows from the recurrence 
relation~\eqref{E:JacobiRecRel}.

\subsection{Legendre polynomials}                          \label{SS:Legendre}

We write $P_k(x):=P_k^{(0)}(x)$ for 
$k=0,1,\ldots$. Clearly, the polynomials $P_k$ are the well-known 
\textit{Legendre polynomials}. 

We will use the following classical Christoffel--Darboux formula for Legendre 
polynomials: 
\begin{equation}\label{E:Chris-Dar}
\displaystyle
\sum_{i=0}^{k}(2i+1)P_i(x)P_i(y)=
           (k+1)\frac{P_k(y)P_{k+1}(x)-P_k(x)P_{k+1}(y)}{x-y},
\end{equation}
where $x\neq y$, and $k\in\mathbb N$.

It is also true that the indefinite integral of the polynomial $P_k$ is given 
by
\begin{equation}\label{E:Int-P_k-1}
\int P_{k}(x)\,\mbox{d}x=\frac{1}{2k+1}\left(P_{k+1}(x)-P_{k-1}(x)\right)
+\mbox{const.},
\end{equation}
where $k\geq0$. Furthermore, one can check that
\begin{equation}\label{E:Int-P_k-2}
\int P_{k}(x)\,\mbox{d}x=
      \frac{1}{k(k+1)}(x^2-1)\frac{\mbox{d}}{\mbox{d}x}P_{k}(x)+\mbox{const.}\\
\end{equation}
for $k\geq1$.

\begin{assumption}\label{A:Assumption-1}
From now on, let $n\in\mathbb N$ be fixed. Let $\tau_i\equiv\tau^{(n)}_i$ 
$(i=1,2,\ldots,n;\ n>0)$,
\begin{equation}\label{E:Def-tau_i}
-1<\tau_1<\tau_2<\cdots<\tau_n<1,
\end{equation}
be the \textit{zeros} of the $n$th Legendre polynomial. We have 
$\tau_i=-\tau_{n+1-i}$ $(i=1,2,\ldots,n)$.

In this paper, we assume that the zeros $\tau^{(n)}_i$ $(i=1,2,\ldots,n)$ are
known with high numerical accuracy, since they can be precomputed using, for 
example, the results presented in~\cite{HT2013}, \cite{JM2018}. 
\end{assumption}

\section{Gauss--Legendre polynomials and curves}            \label{S:GL-curves}

The GL polynomials and the GL curves were recently introduced 
in~\cite{Moon2023}. We now recall their definitions, as well as some of their 
main properties proved in that paper.

The GL polynomials $F^n_i\in\Pi_{n}\setminus\Pi_{n-1}$ $(i=0,1,\ldots,n)$ are 
given by
\begin{equation}\label{E:Def-F^n_i}
F^n_i(t):=G^n_i(t)-G^n_{i+1}(t),
\end{equation}
where $G^n_0(t)=-G^n_{n+1}(t)\equiv\frac{1}{2}$, and
\begin{equation}\label{E:Def-G^n_i}
G^n_i(t):=\frac{nP_{n-1}(\tau_i)}{2}
               \int_{-1}^{t}\frac{P_n(x)}{x-\tau_i}\,\mbox{d}x-\frac{1}{2},
\end{equation}
for $i=1,2,\ldots,n$.

Polynomials $F^n_0, F^n_1,\ldots, F^n_n$ form a basis of $\Pi_n$. They also 
give the \textit{partition of unity}, i.e.,
$$
\sum_{i=0}^{n}F^n_i(t)\equiv 1\qquad(t\in\mathbb R).
$$
We have
\begin{equation}\label{E:F^n_i-delta}
F^n_i(-1)=\delta_{0i},\qquad F^n_i(1)=\delta_{ni}\qquad (i=0,1,\ldots,n).
\end{equation}
Moreover, the following symmetries hold:
\begin{equation}\label{E:F-symm}
F^n_i(-t)=F^{n}_{n-i}(t)\quad (0\leq i\leq n),
\end{equation}
because $G^n_i(-t)=-G^{n}_{n-i+1}(t)$ for $0\leq i\leq n+1$.

The GL curve $\p{P}_n:[-1,1]\to\mathbb E^d$ $(d\geq1)$ of \textit{degree} $n$ 
is a parametric polynomial curve of the form
\begin{equation}\label{E:Def-GL-curve}
\p{P}_n(t):=\sum_{i=0}^{n}\p{W}_iF^n_i(t),
\end{equation}
where $\p{W}_0,\p{W}_1,\ldots,\p{W}_n\in\mathbb E^d$ are the 
so-called \textit{control points}. 

Note that $\p{P}_n(-1)=\p{W}_0$, $\p{P}_n(1)=\p{W}_n$ 
(cf.~\eqref{E:F^n_i-delta}). Although the GL curve does not lie in the convex 
hull of its control points, its graph closely follows its control polygon. This 
property has, in recent years, contributed to the growing popularity of these 
curves.

\section{New representations for GL polynomials}              \label{S:NewRep}

As mentioned in~\S\ref{S:Introduction}, the problem of evaluating GL curves was 
not considered in~\cite{Moon2023}. The purpose of this article is therefore to 
propose fast and numerically stable methods for computing points on a GL 
curve~\eqref{E:Def-GL-curve}.

In view of Eq.~\eqref{E:Def-F^n_i}, the main problem is to efficiently evaluate 
$G^n_i(t)$ for a given $t\in[-1,1]$ and for all $i=1,2,\ldots,n$ 
(cf.~\eqref{E:Def-G^n_i}). To this end, we study the representations of 
polynomials $G^n_i$ in \textit{(i)} shifted power basis, and \textit{(ii)}
in symmetric Jacobi basis.

\subsection{Shifted power basis}                         \label{SS:PowerBasis}

Using~\eqref{E:P_k-symm} and~\eqref{E:P_k-power}, we obtain
\begin{equation}\label{E:P_n-power}
P_n(x)=\sum_{k=0}^na_k(x+1)^k,
\end{equation}
where
$$
a_{k}:=\frac{(-1)^{n-k}}{2^k}\binom{n}{k}\binom{n+k}{k}\qquad (0\leq k\leq n).
$$
Observe that
\begin{equation}\label{E:a-RecRel}
2(k+1)^2a_{k+1}+(n-k)(n+k+1)a_{k}=0,
\end{equation}
where $k=n-1,n-2,\ldots,0$, and 
$$
a_{n}=\frac{1}{2^n}\binom{2n}{n}.
$$

\begin{lemma}\label{L:RecRel-power}
Let the coefficients $b^{(i)}_k$ $(k=0,1,\ldots,n-1;\, 1\leq i\leq n)$ be such 
that
\begin{equation}\label{E:P_n/(x-tau_i)-power}
\frac{P_n(x)}{x-\tau_i}=\sum_{k=0}^{n-1}b^{(i)}_k(x+1)^k
\end{equation}
(cf.~\eqref{E:Def-G^n_i}).

The following non-homogeneous first-order recurrence relation holds:
\begin{equation}\label{E:b-RecRel-I}
b^{(i)}_{k-1}-(\tau_i+1)b^{(i)}_{k}=a_{k}\qquad (k=n,n-1,\ldots,1),
\end{equation}
where $b^{(i)}_{n}:=0$.

In addition, the coefficients $b^{(i)}_k$ satisfy the second-order homogeneous 
recurrence relation of the form
\begin{equation}\label{E:b-RecRel-II}
(n-k)(n+k+1)b^{(i)}_{k-1}+(2(k+1)^2-(\tau_i+1)(n-k)(n+k+1))b^{(i)}_{k}
                                           -2(\tau_i+1)(k+1)^2b^{(i)}_{k+1}=0,
\end{equation}
where $k=n-1,n-2,\ldots,1$, $b^{(i)}_{n-1}=a_{n}$, and $b^{(i)}_n:=0$.
\end{lemma}
\begin{proof}
We have
\begin{eqnarray*}
P_n(x)&=&(x-\tau_i)\sum_{k=0}^{n-1}b^{(i)}_k(x+1)^k=
         (x+1-(\tau_i+1))\sum_{k=0}^{n-1}b^{(i)}_k(x+1)^k\\
      &=&\sum_{k=0}^{n-1}b^{(i)}_k(x+1)^{k+1}-
         (\tau_i+1)\sum_{k=0}^{n-1}b^{(i)}_k(x+1)^k\\
      &=&\sum_{k=0}^{n}
               \left(b^{(i)}_{k-1}-(\tau_i+1)b^{(i)}_{k}\right)(x+1)^k,
\end{eqnarray*}
where $b^{(i)}_{-1}=b^{(i)}_{n}:=0$. It follows that 
the recurrence~\eqref{E:b-RecRel-I} holds.

Combining~\eqref{E:b-RecRel-I} with~\eqref{E:a-RecRel} yields 
the recurrence~\eqref{E:b-RecRel-II}.
\end{proof}

\begin{remark}\label{R:Rem-power-coeffs}
Solving the recurrence~\eqref{E:b-RecRel-I} gives an explicit expression 
for the coefficients~$b^{(i)}_{k}$,
$$
b^{(i)}_{k}=a_{k+1}
                \hyper{3}{2}{k-n+1,\,n+k+2,\,1}{k+2,\,k+2}{\frac{\tau_i+1}{2}}
$$
(cf.~\eqref{E:P_n-power}), where $k=0,1,\ldots,n-1$.
\end{remark}

Now, it is easy to integrate the expansion~\eqref{E:P_n/(x-tau_i)-power} over 
$x\in[-1,t]$ $(-1\leq t\leq 1)$ and obtain a~representation of the polynomials
$F^n_i$ $(i=0,1,\ldots,n)$ in the shifted power basis.

Summarizing, the following statements are true.

\begin{theorem}\label{T:F^n_i-power}
The GL polynomial $F^n_i$ (cf.~\eqref{E:Def-F^n_i}) has the following 
representation in the shifted power basis:
$$
F^n_i(t)=\sum_{k=1}^{n}c^{(i)}_{k}(t+1)^k,
$$
where $i=1,\ldots,n-1$, and
\begin{equation}\label{E:Def-c^i_k}
c^{(i)}_{k}:=\frac{n}{2k}
                 \left(P_{n-1}(\tau_i)b^{(i)}_{k-1}-
                       P_{n-1}(\tau_{i+1})b^{(i+1)}_{k-1}\right)
                                                       \qquad (1\leq k\leq n).
\end{equation}
We also have
\begin{eqnarray*}
&&F^n_0(t)=1-\frac{1}{2}nP_{n-1}(\tau_1)
                \sum_{k=1}^{n}\frac{1}{k}b^{(1)}_{k-1}(t+1)^k,\\
&&F^n_n(t)=\frac{1}{2}nP_{n-1}(\tau_n)
                \sum_{k=1}^{n}\frac{1}{k}b^{(n)}_{k-1}(t+1)^k.
\end{eqnarray*}
\end{theorem}

Hence, for fixed $0\leq i\leq n$, $F^n_i(t)$ $(t\in\mathbb R)$ can be evaluated 
using Horner's scheme (see, e.g., \cite{WG}) in $O(n)$ time 
(cf.~recurrences~\eqref{E:JacobiRecRel}, \eqref{E:a-RecRel} and 
Lemma~\ref{L:RecRel-power}).

\begin{corollary}\label{C:Cor1}
For a given $t\in(-1,1)$ (cf.~\eqref{E:F^n_i-delta}), all the values
$$
F^n_0(t),F^n_1(t),\ldots,F^n_n(t)
$$
can be evaluated in $O(n^2)$ time using Horner's schemes for representations 
given in Theorem~\ref{T:F^n_i-power}. 
\end{corollary}

Note that the cost of the computations can be reduced using the 
symmetries~\eqref{E:P_k-symm}, \eqref{E:F-symm}. This is discussed in 
Section~\ref{SS:FirstApproach}.

Thus, using the approach described in this paragraph, a point on the GL 
curve~\eqref{E:Def-GL-curve} in $\mathbb E^d$ can be computed component-wise
with a computational complexity of $O(n^2+dn)$.

\subsection{Symmetric Jacobi basis}                 \label{SS:SymmJacobiBasis}

In this paragraph, we show that, using the Christoffel--Darboux 
identity~\eqref{E:Chris-Dar} and other results presented in 
Section~\ref{S:Jacobi}, one can derive simple expansions of the polynomials 
$G^n_i$ $(i=1,\ldots,n)$ closely related to the symmetric Jacobi bases. Note 
that the coefficients in those representations are given in terms of the values 
of Legendre polynomials at the zeros $\tau_1,\tau_2,\ldots,\tau_n$ of $P_n$.

Let us define
\begin{equation}\label{E:Def-R_k}
R_k(t):=P_{k+1}(t)-P_{k-1}(t)\qquad (k=0,1,\ldots).
\end{equation}
   
\begin{lemma}\label{L:G^n_i-Jacobi-exp}
For $i=1,2,\ldots,n$, the following identities are true:
\begin{equation}\label{E:G^n_i-Legendre}
G^n_i(t)=-\frac{1}{2}\sum_{k=0}^{n-1}R_k(\tau_i)P_k(t)
                                           +\frac{1}{2}P_{n-1}(\tau_i)P_n(t).
\end{equation}
\begin{equation}\label{E:G^n_i-P1}
G^n_i(t)=\frac{t}{2}+\frac{t^2-1}{2}
             \sum_{k=1}^{n-1}\frac{2k+1}{2k}P_k(\tau_i)P^{(1)}_{k-1}(t).
\end{equation}
\end{lemma}
\begin{proof}
From the recurrence relation~\eqref{E:JacobiRecRel} for $k:=n+1$, $\alpha:=0$,
and $x:=\tau_i$ (cf.~\eqref{E:Def-tau_i}), it follows that
\begin{equation}\label{E:Rel-tau_i}
(n+1)P_{n+1}(\tau_i)=-nP_{n-1}(\tau_i),
\end{equation}
because $P_n(\tau_i)=0$. From the Christoffel--Darboux 
formula~\eqref{E:Chris-Dar} with $k:=n$ and $y:=\tau_i$, we have
\begin{equation}\label{E:Chris-Dar-tau_i}
\displaystyle
\sum_{k=0}^{n}(2k+1)P_k(\tau_i)P_k(x)=
           -(n+1)P_{n+1}(\tau_i)\frac{P_n(x)}{x-\tau_i}=
             nP_{n-1}(\tau_i)\frac{P_n(x)}{x-\tau_i},
\end{equation}
where we used~\eqref{E:Rel-tau_i}. Here $i=1,2,\ldots,n$.

Integrating the right-hand side of~\eqref{E:Chris-Dar-tau_i} over $x\in[-1,t]$ 
with $t\in[-1,1]$ and using~\eqref{E:Int-P_k-1}, we get~\eqref{E:G^n_i-Legendre}, 
because $P_k(-1)=(-1)^k$ for $k\in\mathbb N$ (cf.~\eqref{E:P_n-power}).

Applying~\eqref{E:Jacobi-diff} for $\alpha:=0$, $i:=1$ to~\eqref{E:Int-P_k-2}, 
we obtain
$$
\int P_{k}(x)\,\mbox{d}x=\frac{1}{2k}(x^2-1)P^{(1)}_{k-1}(x)+\mbox{const.}
                                                               \qquad (k\geq1)
$$
(cf.~\eqref{E:Def-JacobiSymm}). Integrating as above and using the latter, 
we derive the identity~\eqref{E:G^n_i-P1}.
\end{proof}

We are now ready to formulate the main results of this paragraph.

\begin{theorem}\label{T:F^n_i-Jacobi-0}
The GL polynomial $F^n_i$ (cf.~\eqref{E:Def-F^n_i}) has the following 
representation in the Legendre basis:
\begin{equation*}
F^n_i(t)=\frac{1}{2}\left(P_{n-1}(\tau_i)-P_{n-1}(\tau_{i+1})\right)P_n(t)+
                \frac{1}{2}\sum_{k=0}^{n-1}
                           \left(R_{k}(\tau_{i+1})-R_{k}(\tau_i)\right)P_k(t),
\end{equation*}
where $i=1,\ldots,n-1$. We also have
\begin{eqnarray*}
&&F^n_0(t)=\frac{1}{2}-\frac{1}{2}P_{n-1}(\tau_1)P_n(t)+
                      \frac{1}{2}\sum_{k=0}^{n-1}R_{k}(\tau_1)P_k(t),\\
&&F^n_n(t)=\frac{1}{2}+\frac{1}{2}P_{n-1}(\tau_n)P_n(t)-
                      \frac{1}{2}\sum_{k=0}^{n-1}R_{k}(\tau_n)P_k(t).
\end{eqnarray*}
\end{theorem}

\begin{theorem}\label{T:F^n_i-Jacobi-1}
The GL polynomial $F^n_i$ can be written as:
\begin{equation*}
F^n_i(t)=\frac{1}{2}(t^2-1)\left(
             \sum_{k=1}^{n-1}\frac{2k+1}{2k}
                       (P_k(\tau_i)-P_k(\tau_{i+1}))P^{(1)}_{k-1}(t)\right)
\end{equation*}
(cf.~\eqref{E:Def-JacobiSymm}), where $i=1,\ldots,n-1$. It is also true that
\begin{eqnarray*}
&&F^n_0(t)=1-\frac{1}{2}(t+1)\left(1+(t-1)
          \sum_{k=1}^{n-1}\frac{2k+1}{2k}P_k(\tau_1)P^{(1)}_{k-1}(t)\right),\\
&&F^n_n(t)=\frac{1}{2}(t+1)\left(1+(t-1)
             \sum_{k=1}^{n-1}\frac{2k+1}{2k}P_k(\tau_n)P^{(1)}_{k-1}(t)\right).
\end{eqnarray*}
\end{theorem}

Observe that the expansions given in Theorem~\ref{T:F^n_i-Jacobi-1} immediately 
imply~\eqref{E:F^n_i-delta}. 
By using both results, the recurrence relation~\eqref{E:JacobiRecRel} and the 
Clenshaw algorithm (see~\S\ref{S:Jacobi}, p.~\pageref{E:Clenshaw}), the value 
$F^n_i(t)$ for a~given $t\in\mathbb R$ can be computed with $O(n)$ 
computational complexity.

\begin{corollary}\label{C:Cor2}
Let $t\in(-1,1)$ (cf.~\eqref{E:F^n_i-delta}). All the values
$$
F^n_0(t),F^n_1(t),\ldots,F^n_n(t)
$$
can be evaluated in $O(n^2)$ time using the Clenshaw algorithm for the 
representations given in Theorems~\ref{T:F^n_i-Jacobi-0} 
or~\ref{T:F^n_i-Jacobi-1}. 
\end{corollary}

As noted at the end of \S\ref{SS:PowerBasis}, the computational cost can be 
decreased using the symmetries~\eqref{E:P_k-symm}, \eqref{E:F-symm}. This is 
discussed in Section~\ref{SS:SecondApproach}, where we make use of 
Theorem~\ref{T:F^n_i-Jacobi-1}.

In summary, using the methods proposed in this paragraph, one can evaluate 
a point on the GL curve~\eqref{E:Def-GL-curve} in 
$\mathbb E^d$ component-wise with $O(n^2+dn)$ computational complexity.

\begin{remark}
Note that the expansions given in Theorem~\ref{T:F^n_i-Jacobi-0} were obtained 
independently by Ramanantoanina and Hormann in~\cite{RH2026}. In that article, 
a method for evaluating and rendering GL curves was proposed. The method is 
based on the Legendre representation of the GL polynomials and also has 
computational complexity $O(n^2+dn)$. The main difference is that we employ the 
Clenshaw algorithm, whereas Ramanantoanina and Hormann directly use the 
recurrence relation for the Legendre polynomials.
\end{remark}

\subsection{Derivatives}                                \label{SS:Derivatives}

In~\cite[Theorem 16]{Moon2023} a complicated general formula for the $m$th 
derivative of a GL polynomial $F^n_i$ was derived, where the derivative is 
expressed in terms of \textit{elementary symmetric polynomials}. 

It is now clear that, using the representations of GL polynomials 
from~\S\ref{SS:PowerBasis} and~\S\ref{SS:SymmJacobiBasis}, we can 
find simple expressions for their $m$th derivatives.

For example, if $1\leq i\leq n-1$ and $m>0$, we have
\begin{equation}\label{E:F^n_i-diff-1}
\frac{\mbox{d}^m}{\mbox{d}t^m}F^n_i(t)=
                             \sum_{k=m}^{n}(k-m+1)_mc^{(i)}_{k}(t+1)^{k-m}
\end{equation}
(cf.~Theorem~\ref{T:F^n_i-power}) and, in particular, 
$$
\frac{\mbox{d}^m}{\mbox{d}t^m}F^n_i(-1)=
            (-1)^m\frac{\mbox{d}^m}{\mbox{d}t^m}F^n_{n-i}(1)=
                                      m!c^{(i)}_{m}\qquad (i=1,2,\ldots,n-1).
$$ 
On the other hand, using~\eqref{E:Jacobi-diff} for $\alpha:=0$ yields the 
following formula:
\begin{equation}\label{E:F^n_i-diff-2}
\frac{\mbox{d}^m}{\mbox{d}t^m}F^n_i(t)=
    \frac{(n+1)_m}{2^{m+1}}\left(P_{n-1}(\tau_i)-P_{n-1}(\tau_{i+1})\right)
                                                            P_{n-m}^{(m)}(t)+
           \frac{1}{2^{m+1}}\sum_{k=m}^{n-1}(k+1)_md^{(i)}_kP_{k-m}^{(m)}(t),
\end{equation}
where $d^{(i)}_k:=R_{k}(\tau_{i+1})-R_{k}(\tau_i)$
(see~Theorem~\ref{T:F^n_i-Jacobi-0} and~\eqref{E:Def-R_k}).

Note that the expression~\eqref{E:F^n_i-diff-1} can be evaluated in $O(n-m)$ 
time using Horner's scheme. Similarly, for a given $t\in\mathbb R$, 
Eq.~\eqref{E:F^n_i-diff-2} can be computed with the same computational 
complexity by applying the recurrence~\eqref{E:JacobiRecRel} and the Clenshaw 
algorithm (cf.~\S\ref{S:Jacobi}, p.~\pageref{E:Clenshaw}).

Certainly, the derivatives of $F^n_0$ and $F^n_n$ can be obtained in a similar 
way.

\section{Evaluation}                                      \label{S:Evaluation}

We now turn to the problem of evaluating a GL curve $\p{P}_n$ of degree $n$ in 
$\mathbb E^d$ (see~\eqref{E:Def-GL-curve}).

More precisely, based on the results presented in~\S\ref{S:NewRep}, we 
propose two methods for computing
$$
F^n_0(t), F^n_1(t),\ldots, F^n_n(t)
$$
for a fixed $t\in(-1,1)$ (cf.~\eqref{E:F^n_i-delta}), which run with $O(n^2)$ 
computational complexity. Once all these values are known, one can determine 
a point $\p{P}_n(t)$ component-wise in $O(nd)$ time. This results in a total 
computational complexity of $O(n^2+dn)$ (see Corollaries~\ref{C:Cor1}, 
	\ref{C:Cor2}).

Moreover, we present a sketch of the algorithms for the evaluation of $M$ 
points on the curve $\p{P}_n$, which work in $O(Mdn+dn^2)$ time.

First of all, let us note that there is a possibility of reducing the 
computational cost by using the symmetries~\eqref{E:F-symm}. Indeed, it is 
enough to consider the problem of evaluating 
$$
F^n_0, F^n_1,\ldots, F^n_{\floor{n/2}},
$$
because we have
\begin{equation}\label{E:P_n-symm}
\p{P}_n(t)=\sum_{i=0}^{\floor{n/2}}\p{W}_iF^n_i(t)+
              \sum_{i=\floor{n/2}+1}^{n}\p{W}_iF^n_{n-i}(-t).
\end{equation}

Secondly, it is also worth making use of the following simple observation. For 
odd $n=2m+1$, one of the zeros of $P_n$ is equal to~$0$. Then $n-1$ is even and
\begin{equation}\label{E:P_odd-at-zero} 
P_{n-1}(0)=\frac{(-1)^m}{4^m}\binom{2m}{m}.
\end{equation} 

\subsection{First approach: shifted power basis}      \label{SS:FirstApproach}

Let us use the results from \S\ref{SS:PowerBasis}.

First, we need to compute
$$
P_{n-1}(\tau_i)\qquad (i=1,2,\ldots,\ceil{n/2})
$$
(cf.~Assumption~\ref{A:Assumption-1}, and~\eqref{E:P_odd-at-zero}). This can be 
accomplished by the recurrence relation~\eqref{E:JacobiRecRel} with $\alpha:=0$ 
in $O(\frac{1}{2}n^2)$ time.

In the second step, we determine the coefficients $b^{(i)}_k$ 
(see~\eqref{E:P_n/(x-tau_i)-power}) for all $i=1,2,\ldots,\ceil{n/2}$ 
(cf.~\eqref{E:Def-c^i_k}) and all $k=0,1,\ldots,n-1$. It should be done with 
the help of the recurrence relations~\eqref{E:b-RecRel-I} 
and~\eqref{E:a-RecRel}.

Finally, for a given $t\in(-1,1)$, we obtain the values
\begin{equation}\label{E:F^n_i_in_+t_and_-t}
F^n_0(\pm t), F^n_1(\pm t),\ldots, F^n_{\floor{n/2}}(\pm t)
\end{equation}
by using Horner's schemes for expansions given in 
Theorem~\ref{T:F^n_i-power} and thus we also find component-wise the point 
$\p{P}_n(t)$ (cf.~\eqref{E:P_n-symm}).


Note that this approach has $O(n^2+dn)$ computational complexity when 
computing a~single point on an $n$th-degree GL curve in $\mathbb E^d$. 


\subsection{Second approach: symmetric Jacobi basis} \label{SS:SecondApproach}

Now, we present another technique for evaluating GL curves, which uses the
representations from Theorem~\ref{T:F^n_i-Jacobi-1}. Note that the same can be 
done for the representations from Theorem~\ref{T:F^n_i-Jacobi-0}.

To this end, in the first step, we have to find the following values of the 
Legendre polynomials:
$$
P_k(\tau_i)\qquad 
       \mbox{for all $k=1,2,\ldots,n-1$ and $i=1,2,\ldots,\ceil{n/2}$}.
$$
Again, this can be done using the recurrence relation~\eqref{E:JacobiRecRel} 
with $\alpha:=0$ in $O(\frac{1}{2}n^2)$ computational complexity 
(cf.~\eqref{E:P_odd-at-zero}). 

Next, for a fixed $-1<t<1$, we obtain all the 
values~\eqref{E:F^n_i_in_+t_and_-t} using the Clenshaw algorithms with 
$\alpha:=1$ (see~\S\ref{S:Jacobi}, p.~\pageref{E:Clenshaw}) for the
expansions given in Theorem~\ref{T:F^n_i-Jacobi-1}.


The presented method also has a computational complexity of $O(n^2+dn)$ when 
computing a single point on a GL curve of degree $n$ in $\mathbb E^d$. 


\subsection{Multipoint evaluation}                       \label{SS:MultiEval}

Let us now focus on the multipoint evaluation of the GL curve, as this problem
is essential for its rendering. Suppose we want to compute
\begin{equation}\label{E:MultiPoint}
\p{P}_n(t_0), \p{P}_n(t_1),\ldots, \p{P}_n(t_M), 
\end{equation}
where $-1\leq t_0<t_1<\cdots<t_{M}\leq 1$ $(M\in\mathbb N)$, and 
$\p{P}_n:[-1,1]\to\mathbb E^d$ is a GL curve of degree~$n$ 
(see~\eqref{E:Def-GL-curve}). Since $M\gg n$ and $d$ is small in practical 
applications, we will consider this case in the sequel.

If we use results from~\S\ref{SS:FirstApproach} and~\S\ref{SS:SecondApproach},
we can find all the points~\eqref{E:MultiPoint} in $O(M(n^2+dn))=O(Mn^2)$ time.

Let us observe that it is possible to find these points much faster, meaning
with a computational complexity of $O(Mdn+n^2)=O(Mdn)$.

\begin{remark}\label{R:Rem-1}
For clarity, we sketch a solution which does not use the 
symmetries~\eqref{E:F-symm} (it means the form~\eqref{E:P_n-symm}), as this 
better presents the main idea of fast multipoint evaluation.
\end{remark}

As shown above, we can efficiently compute all the coefficients 
$c^{(i)}_k$ $(i=0,1,\ldots,n;\, k=0,1,\ldots,n)$ such that
$$
F^n_i(t)=\sum_{k=0}^{n}c^{(i)}_k(t+1)^k\qquad (0\leq i\leq n)
$$
(cf.~Theorem~\ref{T:F^n_i-power}) with $O(n^2)$ computational complexity. 

For simplicity, let us suppose that $d=1$ (for higher dimensions of the curve, 
the computations can be done component-wise with the same approach). Set
\begin{equation}\label{E:Def-p_m}
p_n(t):=\sum_{i=0}^{n}w_iF^n_i(t)\qquad (w_0,w_1,\ldots,w_n\in\mathbb R).
\end{equation}
We have
$$
p_n(t)=\sum_{k=0}^{n}q_k(t+1)^k,\qquad
\mbox{where}
\qquad
q_k:=\sum_{i=0}^{n}w_ic^{(i)}_k.
$$
Then all the coefficients $q_k$ $(0\leq k\leq n)$ are independent of $t$ and
they can be computed in $O(n^2)$ time. Next, all the values
\begin{equation}\label{E:MultiPoints-2}
p_n(t_0), p_n(t_1),\ldots, p_n(t_M)
\end{equation}
can be evaluated using Horner's scheme $M+1$ times with $O(Mn)$ total 
computational complexity.

A similar improvement is possible if we prefer the representations given in
Theorem~\ref{T:F^n_i-Jacobi-0} or Theorem~\ref{T:F^n_i-Jacobi-1}. 
For example, as observed above, we can determine all the coefficients 
$g^{(i)}_k$ $(i=0,1,\ldots,n;\, k=0,1,\ldots,n-2)$ such that
\begin{eqnarray*}
&&F^n_0(t)=\frac{1-t}{2}+
                   \frac{t^2-1}{2}\sum_{k=0}^{n-2}g^{(0)}_kP_k^{(1)}(t),\\
&&F^n_i(t)=\frac{t^2-1}{2}\sum_{k=0}^{n-2}g^{(i)}_kP_k^{(1)}(t)
                                                 \qquad (1\leq i\leq n-1),\\ 
&&F^n_n(t)=\frac{1+t}{2}+
                   \frac{t^2-1}{2}\sum_{k=0}^{n-2}g^{(n)}_kP_k^{(1)}(t)
\end{eqnarray*}
in $O(n^2)$ time. Thus, we have
\begin{eqnarray*}
p_n(t)&=&\frac{1-t}{2}w_0+\frac{1+t}{2}w_n+
               \frac{t^2-1}{2}\sum_{i=0}^{n}w_i
                                 \sum_{k=0}^{n-2}g^{(i)}_kP^{(1)}_{k}(t)\\
      &=&\frac{1-t}{2}w_0+\frac{1+t}{2}w_n+
               \frac{t^2-1}{2}\sum_{k=0}^{n-2}s_kP^{(1)}_{k}(t)                                 
\end{eqnarray*}
(cf.~\eqref{E:Def-p_m}), where
$$
s_k:=\sum_{i=0}^{n}w_ig^{(i)}_k.
$$
The coefficients $s_k$ $(0\leq k\leq n-2)$ are independent of $t$ and they can 
be found with a computational complexity of $O(n^2)$. After that, all the 
values~\eqref{E:MultiPoints-2} can be evaluated using the Clenshaw algorithm 
$M+1$ times in $O(Mn)$ time in total.

In~\ref{A:Appendix}, we provide pseudocode for each of the three proposed 
methods for multipoint evaluation of the GL curve~\eqref{E:Def-GL-curve}.

\begin{remark}\label{R:RemOpt}
Note that the methods for fast multipoint evaluation of GL curves described 
above, as well as the algorithms presented in~\ref{A:Appendix}, can be 
further optimized by exploiting the symmetries~\eqref{E:F-symm}. For example, 
since~\eqref{E:P_n-symm} holds, it suffices to compute only the 
\textit{first half} of the coefficients $g^{(i)}_k$ with respect to $i$ 
(cf.~Remark~\ref{R:Rem-1}). Certainly, similar optimizations are possible for 
shifted power and Legendre bases.
\end{remark}

\begin{remark}
One can also check that
$$
p_n(t) =
\dfrac{w_0 + w_n}{2}
+ \dfrac{w_n - w_0}{2} t
+ \frac{t^2-1}{2}
  \sum_{k=1}^{n-1}
  \left(
      \frac{2k+1}{2k}
      \sum_{i=1}^{n} \left(\Delta w_{i-1}\right) P_k(\tau_i)\right)
  P^{(1)}_{k-1}(t),
$$
where $\Delta z_i:=z_{i+1}-z_{i}$. The sum over $i$ can be shortened roughly by 
a half if we take into account the symmetry of Jacobi polynomials 
(cf.~\eqref{E:P_k-symm}) and the fact that $\tau_i=-\tau_{n+1-i}$ 
$(i=1,2,\ldots,n)$. This gives, for even $n$,
\begin{multline*}
p_n(t) =
\dfrac{w_0 + w_n}{2}
+ \dfrac{w_n - w_0}{2} t
+ \\
  \frac{t^2-1}{2}
  \sum_{k=1}^{n-1}
  \left(
      \frac{2k+1}{2k}
      \sum_{i=1}^{\lfloor n/2 \rfloor} P_k(\tau_i) \left[\Delta w_{i-1}
      + (-1)^k \Delta w_{n-i} \right] \right)
  P^{(1)}_{k-1}(t),
\end{multline*}
and for odd $n$,
\begin{multline*}
 p_n(t) =
\dfrac{w_0 + w_n}{2}
+ \dfrac{w_n - w_0}{2} t
+ \\
 \frac{t^2-1}{2}
  \sum_{k=1}^{n-1}
  \left(
      \frac{2k+1}{2k}
      \sum_{i=1}^{\lfloor n/2 \rfloor} P_k(\tau_i) \left[\Delta w_{i-1}
      + (-1)^k \Delta w_{n-i} \right] 
      +
      P_k(0) \Delta w_{(n-1)/2}
      \right)
  P^{(1)}_{k-1}(t).
\end{multline*}
The values $\left[\Delta w_{i-1} \pm \Delta w_{n-i} \right]$
can be precomputed and reused for every $k$. The additional summand for odd 
$n$ vanishes for odd $k$ (cf.~also~\eqref{E:P_odd-at-zero}). An analogous 
improvement can be made for both shifted power and Legendre bases.
\end{remark}

\section{Numerical experiments}                                \label{S:Tests}

We have tested the multipoint GL curve evaluation strategies
(cf.~\S\ref{SS:MultiEval}) based on the shifted power
(Theorem~\ref{T:F^n_i-power}), Legendre (Theorem~\ref{T:F^n_i-Jacobi-0}),
and symmetric Jacobi (Theorem~\ref{T:F^n_i-Jacobi-1}) representations,
employing the symmetries~\eqref{E:F-symm} (cf.~Remark~\ref{R:RemOpt}). We have
compared them with the naive direct evaluation of the formulas given
in~\cite{Moon2023} (cf.~the integral representations~\eqref{E:Def-F^n_i},
\eqref{E:Def-G^n_i}).

All evaluations were done using \textsf{Python 3.14.4} with the
\textsf{NumPy 2.5.3} library on a computer with an 
\texttt{Intel Core i5-10600 CPU} at \texttt{3.30GHz} processor and 
\texttt{16GB} of \texttt{RAM}. We have used the double (\texttt{float64}) 
precision. To evaluate the integrals~\eqref{E:Def-G^n_i}, we used the 
\textsf{SciPy 1.17.1} function \texttt{quad}. The roots of Legendre polynomials 
were obtained using the \texttt{roots\_legendre} method from the \textsf{SciPy} 
library.

The source code which was used to perform the tests is available at
\url{https://github.com/filipchudy/gauss-legendre-curve-evaluation}.

\subsection{Numerical performance}                           \label{SS:NumPer}

We have tested the numerical performance of the presented methods for degrees
up to 100, with the 2D control points sampled from a uniform distribution over
$[-1, 1]^2$. We have observed that the results using the explicit integral
formulas~\eqref{E:Def-F^n_i}, \eqref{E:Def-G^n_i} from~\cite{Moon2023} and in
the symmetric Jacobi and the Legendre bases, closely followed each other,
with no significant numerical advantage to either of them.
However, the computations in the shifted power basis have started losing their
precision for $n\gtrsim 20$ (see Figure~\ref{F:Figure2}).

\begin{figure*}[ht!]
\centering
\begin{subfigure}[b]{0.45\textwidth}
        \centering
        \includegraphics[angle=0,width=\textwidth]{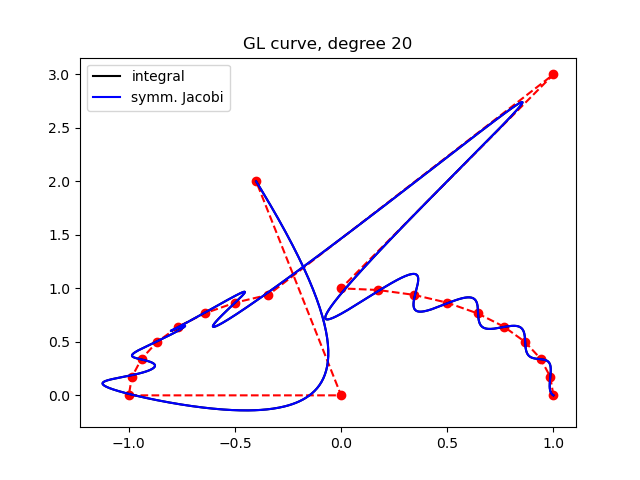}%
        \label{F:Figure2a}
        \subcaption{A predefined GL curve of degree 20 rendered using the 
                    algorithm which uses the Jacobi basis.}
    \end{subfigure}
\begin{subfigure}[b]{0.45\textwidth}
        \centering
        \includegraphics[angle=0,width=\textwidth]{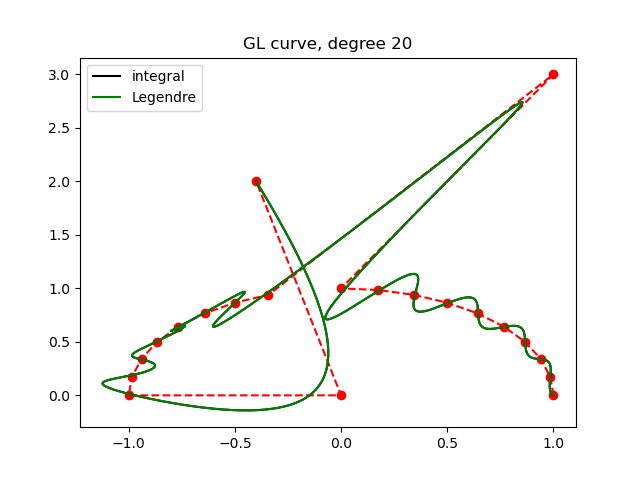}
        \label{F:Figure2b}
        \subcaption{A predefined GL curve of degree 20 rendered using the 
                    algorithm which uses the Legendre basis.}
    \end{subfigure}

\begin{subfigure}[b]{0.45\textwidth}
        \centering
        \includegraphics[angle=0,width=\textwidth]{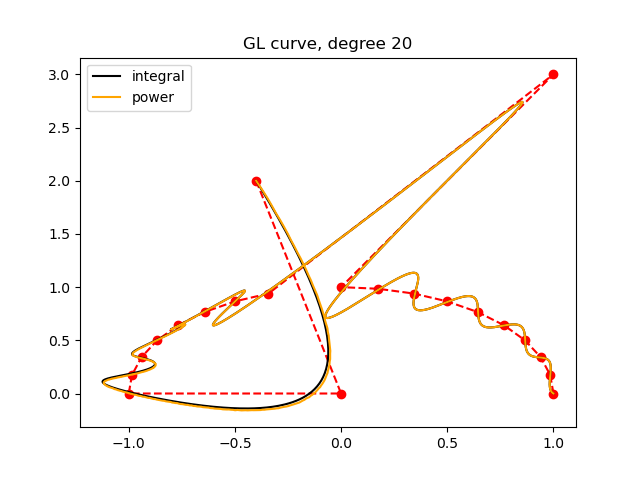}%
        \label{F:Figure2c}
        \subcaption{A predefined GL curve of degree 20 rendered using the 
                    algorithm which uses the power basis.}
    \end{subfigure}
\begin{subfigure}[b]{0.45\textwidth}
    \centering
    \includegraphics[angle=0,width=\textwidth]{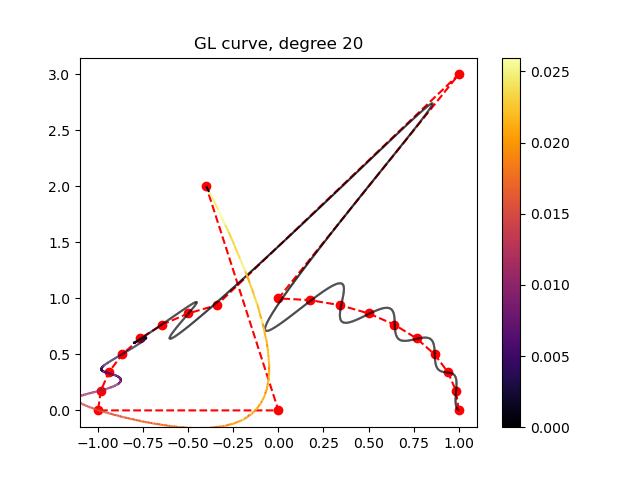}%
        \label{F:Figure2d}
        \subcaption{The $\|\cdot\|_2$ error of the curve rendered using the 
                    power basis.}
    \end{subfigure}
\caption{Numerical instability in the shifted power representation of a GL 
         curve of degree 20.}\label{F:Figure2}
\end{figure*}

This is likely due to the fact that the absolute values of the coefficients 
$a_k$ of Legendre polynomials in the shifted power basis 
(see~\eqref{E:P_n-power}) are very high compared to the absolute values of the 
Legendre polynomials themselves, e.g., for $n=25$,
$$
a_{10} = (-1)^{15}  35! / (10!\;15!\;10!\;1024) \approx - 5.86 \cdot 10^{11},
$$
while $|P_{25}(t)|\leq1$ for $t\in[-1, 1]$. Additionally, the coefficients
$a_k$ have alternating signs, which may result in cancellation errors for
positive $t$. See also Eq.~\eqref{E:Def-c^i_k} and
Remark~\ref{R:Rem-power-coeffs}.

In contrast, the coefficients of $F^n_i$ in the Legendre and symmetric Jacobi
bases have small absolute values. Note that
$$
\frac{1}{2}\left|R_k(\tau_{i+1})-R_k(\tau_{i})\right|\leq\frac{1}{2}\cdot4=2,
\qquad
\frac{2k+1}{4k}\left|P_k(\tau_{i})-P_k(\tau_{i+1})\right|\leq\frac{3}{4}\cdot2=
\frac{3}{2},
$$
where $1\leq k\leq n-1$, because $\max_{x\in[-1,1]}|P_k(x)|\leq 1$ $(k\geq 0)$.
We have checked that for the bases of degrees between 1 and 150, all 
coefficients in both the Legendre and the Jacobi forms have absolute values
lower than 1. See Figure~\ref{F:Figure3}.

\begin{figure*}[ht!]
\centering
\includegraphics[angle=0,width=0.95\linewidth]{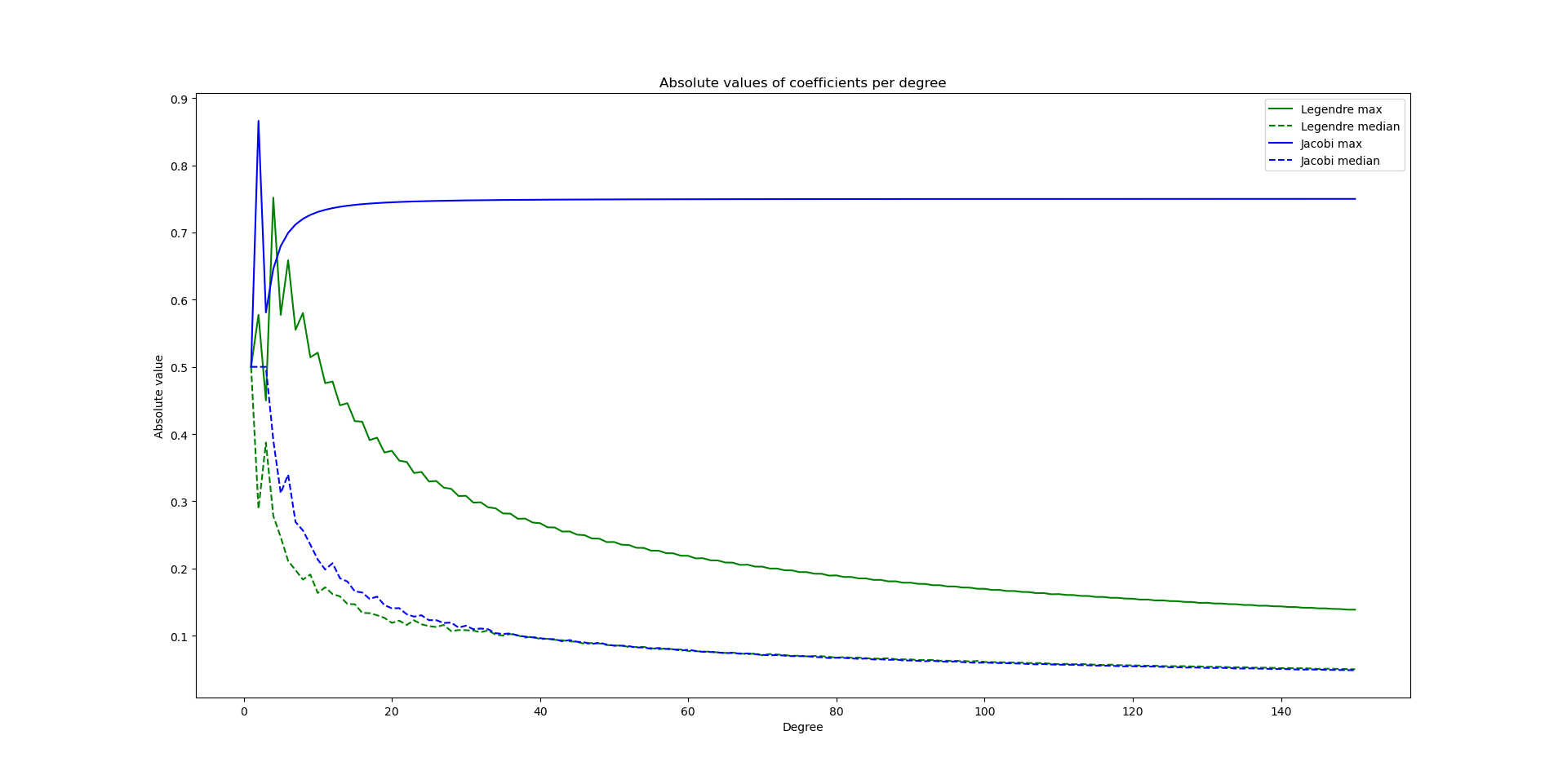}%

\caption{Maximum and median absolute values of the coefficients
of $F^n_i$ in Legendre and Jacobi forms,
for $1 \leq n \leq 150$.}\label{F:Figure3}
\end{figure*}

Taking into account the stability of the Clenshaw algorithm (see, e.g.,
\cite{DJCh2014} and the articles cited therein), we conclude that our methods
perform well in \textsf{fl} arithmetic, provided that the problem of evaluating
a GL curve for given control points is well-conditioned.

For a more detailed table containing mean and maximum $\|\cdot\|_2$
errors when evaluating a GL curve in different bases, see
Table~\ref{T:TableErrors}.

\begin{table*}[ht!]
\begin{center}
\renewcommand{\arraystretch}{1.25}
\begin{tabular}{c|c|c|c|c}
$n$ & error & symm.~Jacobi & Legendre & Power \\ \hline
5   & mean    & 6.953e-15 & 6.946e-15 & 1.783e-14 \\
    & maximum & 4.633e-12 & 4.633e-12 & 4.630e-12 \\ \hline
10  & mean    & 6.176e-14 & 6.176e-14 & 1.232e-10 \\
    & maximum & 1.866e-10 & 1.866e-10 & 1.714e-09 \\ \hline
20  & mean    & 2.117e-12 & 2.117e-12 & 0.003 \\
    & maximum & 2.159e-08 & 2.159e-08 & 0.051 \\ \hline
25  & mean    & 8.171e-13 & 8.171e-13 & 40.078 \\
    & maximum & 2.076e-09 & 2.076e-09 & 414.491 \\ \hline
30  & mean    & 1.360e-12 & 1.360e-12 & 1.153e5 \\
    & maximum & 1.971e-09 & 1.971e-09 & 1.890e6 \\ \hline
40  & mean    & 1.312e-12 & 1.312e-12 & 1.414e13 \\
    & maximum & 4.029e-10 & 4.029e-10 & 1.710e14 \\ \hline
50  & mean    & 1.834e-12 & 1.834e-12 & 7.446e+20 \\
    & maximum & 5.467e-10 & 5.467e-10 & 9.206e+21 \\ \hline
100 & mean    & 2.278e-11 & 2.278e-11 & 6.023e+58 \\
    & maximum & 1.189e-08 & 1.189e-08 & 7.656e+59 \\
\end{tabular}
\renewcommand{\arraystretch}{1}
\vspace{2ex}
\caption{The values of $\| \cdot\|_2$ errors for different degrees and curve 
                                                          evaluation methods.}
\label{T:TableErrors}
\end{center}
\end{table*}

\subsection{Efficiency}                               \label{SS:NumEfficiency}

Despite its poor numerical quality for higher degrees of the GL curve, the
shifted power basis is significantly faster in terms of evaluation times.

Tests have been performed for both low and high degrees of the curves. In both 
cases, for a given degree, a testing set of 100 curves of this degree has been 
generated, with the control points sampled from the uniform distribution over 
$[-1, 1]^2$. Then, each curve has been preprocessed using the approach given 
in~\S\ref{SS:MultiEval}, but with the optimizations pointed out in 
Remark~\ref{R:RemOpt}. Each curve has been evaluated at $t_i=-1+i/2500$ for 
$i=1,2,\ldots,4999$. The methods based on the Legendre, symmetric Jacobi, and 
power basis representations are compared with the naive direct method
based on the integral representations~\eqref{E:Def-F^n_i} 
and~\eqref{E:Def-G^n_i}.

Table~\ref{T:Table1} shows the running times of all four methods for degrees 
$n=1,2,\ldots,15$.

\begin{table*}[ht!]
\begin{center}
\renewcommand{\arraystretch}{1.25}
\begin{tabular}{c|c|c|c|c}
$n$ & symm.~Jacobi & Legendre & Power & Integral \\ \hline
1 & 0.58 & 0.99 & \bf{0.36} & 19.93 \\
2 & 1.15 & 1.57 & \bf{0.55} & 24.56 \\
3 & 1.53 & 1.77 & \bf{0.61} & 35.08 \\
4 & 1.93 & 2.26 & \bf{0.78} & 45.57 \\
5 & 2.29 & 2.63 & \bf{0.89} & 57.55 \\
6 & 2.72 & 2.92 & \bf{0.98} & 67.84 \\
7 & 3.14 & 3.39 & \bf{1.09} & 81.34 \\
8 & 3.44 & 3.77 & \bf{1.20} & 92.63 \\
9 & 3.88 & 4.12 & \bf{1.32} & 103.86 \\
10 & 4.27 & 6.33 & \bf{1.51} & 112.49 \\
11 & 4.64 & 5.00 & \bf{1.60} & 125.61 \\
12 & 5.06 & 5.25 & \bf{1.70} & 140.35 \\
13 & 5.45 & 5.97 & \bf{3.07} & 147.90 \\
14 & 5.79 & 6.14 & \bf{2.51} & 161.31 \\
15 & 8.06 & 6.82 & \bf{2.08} & 170.33 \\
\end{tabular}
\renewcommand{\arraystretch}{1}
\vspace{2ex}
\caption{Times in seconds at low degrees.}%
\label{T:Table1}
\vspace{-3ex}
\end{center}
\end{table*}

For higher degrees, due to the poor numerical properties of the shifted power
basis, we omit the algorithm based on the power basis representation.
Table~\ref{T:Table2} shows the running times of the three remaining methods for
degrees $n=20,25,30,35,40,45,50,100$.

\begin{table*}[ht!]
\begin{center}
\renewcommand{\arraystretch}{1.25}
\begin{tabular}{c|c|c|c}
$n$ & symm.~Jacobi & Legendre & Integral \\ \hline
20 & \bf{8.13} & 10.49 & 229.72 \\
25 & \bf{10.21} & 12.71 & 1116.02 \\
30 & \bf{13.80} & 13.23 & 1871.15 \\
35 & \bf{14.92} & 16.04 & 2793.35 \\
40 & 18.70 & \bf{16.38} & 4091.98 \\
45 & 20.43 & \bf{18.68} & 5325.15 \\
50 & 22.60 & \bf{20.59} & 6605.65 \\
100 & 42.43 & \bf{40.27} & 32307.92 \\
\end{tabular}
\renewcommand{\arraystretch}{1}
\vspace{2ex}
\caption{Times in seconds for higher degrees.}%
\label{T:Table2}
\end{center}
\end{table*}

\subsection{Hybrid method}                                   \label{SS:Hybrid}

In view of the numerical results presented above, we recommend the following
hybrid approach for practical applications:
\begin{enumerate}
\itemsep 1ex

\item if $n\leq 15$, use the method described in~\S\ref{SS:FirstApproach},

\item \label{I:I2}
      if $16\leq n\leq 40$, use the method based on
      Theorem~\ref{T:F^n_i-Jacobi-1} (see also~\S\ref{SS:SecondApproach}),

\item \label{I:I3}
      if $n>40$, use the method based on Theorem~\ref{T:F^n_i-Jacobi-0} 
      (see also~\S\ref{SS:SecondApproach}).
      
\end{enumerate}

Note that switching between these methods does not incur any additional 
computational cost, i.e., when adding or removing a control point, the 
coefficients have to be reevaluated from scratch as there is no known relation 
between the Legendre polynomials' roots of consecutive degrees. A similar 
hybrid approach can be used for multipoint 
evaluation (cf.~\S\ref{SS:MultiEval} and~\ref{A:Appendix}).

\section{Conclusions}                                    \label{S:Conclusions}

In this paper, we have given new representations of Gauss--Legendre 
polynomials and their derivatives in the shifted power basis and in some 
orthogonal bases closely related to Jacobi polynomials. We have also proposed 
efficient methods for evaluating Gauss--Legendre curves, including multipoint 
evaluation. The numerical tests demonstrate that the hybrid method proposed
in~\S\ref{SS:Hybrid} is stable and computationally efficient, even for 
high-degree curves.

Note that we are currently conducting research on the connection between the 
Bernstein and Gauss--Legendre bases, as well as on the construction of the dual 
Gauss--Legendre basis, which may be useful in many applications, for example in 
approximation problems. We intend to publish these results in the near future.  

\section*{Acknowledgements}

The authors would like to thank the three anonymous referees for their valuable 
critical remarks and suggestions. Their comments were very helpful in preparing 
the revised and improved version of the article.

\newpage
\appendix
\renewcommand{\thealgorithm}{\Alph{section}.\arabic{algorithm}}
\section{Algorithms}                                        \label{A:Appendix}

\begin{algorithm}[ht!]
\caption{Multipoint evaluation of a GL curve using a power 
                                                basis}\label{A:MultiEvalPower}
\begin{algorithmic}[1]
\setlength{\itemsep}{3pt}
\Procedure {GLEvalPower}{$n, [t_0, t_1, \ldots, t_M], \p{W}$}

\State $\tau_1, \ldots, \tau_n \gets \texttt{roots\_legendre(n)}$
\State $P_{n-1}(\tau_1), \ldots, P_{n-1}(\tau_{\lceil n/2 \rceil}) 
                         \gets\texttt{evaluate using \eqref{E:JacobiRecRel}}$
\State $P_{n-1}(\tau_{\lceil n/2 \rceil + 1}), \ldots P_{n-1}(\tau_n) 
                               \gets\texttt{use symmetry \eqref{E:P_k-symm}}$

\State $a_0 \gets (-1)^n$
\State $a_1, a_2, \ldots, a_n\gets\texttt{evaluate using \eqref{E:a-RecRel}}$

\For {$i \gets 1$ \textbf{to} $n$}
  \State $b^{(i)}_{n-1} \gets a_n$
  \State $b^{(i)}_{n-2}, b^{(i)}_{n-3}, \ldots, b^{(i)}_{0} 
                           \gets\texttt{evaluate using \eqref{E:b-RecRel-I}}$
\EndFor

\State $c^{(0)}_0 \gets 1$
\For {$k \gets 1$ \textbf{to} $n$}
  \State $c^{(0)}_k \gets -n/(2k) P_{n-1}(\tau_1) b^{(1)}_{k-1}$
\EndFor
\For {$i \gets 1$ \textbf{to} $n-1$}
  \State $c^{(i)}_0 \gets 0$
  \For {$k \gets 1$ \textbf{to} $n$}
    \State $c^{(i)}_k \gets n/(2k) \left( P_{n-1}(\tau_i) b^{(i)}_{k-1} - 
                                P_{n-1}(\tau_{i+1}) b^{(i+1)}_{k-1} \right)$
  \EndFor
\EndFor
\State $c^{(n)}_0 \gets 0$
\For {$k \gets 1$ \textbf{to} $n$}
  \State $c^{(n)}_k \gets n/(2k) P_{n-1}(\tau_n) b^{(n)}_{k-1}$
\EndFor

\For {$k \gets 0$ \textbf{to} $n$}
  \State $\p{Q}_k \gets \sum_{i=0}^n c_k^{(i)} \p{W}_i$
\EndFor

\For {$k \gets 0$ \textbf{to} $M$}
 \State $\p{R}_k 
      \gets\texttt{Horner(}[\p{Q}_0,\p{Q}_1,\ldots,\p{Q}_n], 1+t_k\texttt{)}$
\EndFor

\State \Return $[\p{R}_0, \p{R}_1, \ldots, \p{R}_M]$


\EndProcedure
\end{algorithmic}
\end{algorithm}

\newpage

\begin{algorithm}[ht!]
\caption{Multipoint evaluation of a GL curve using 
                                   a Jacobi basis}\label{A:MultiEvalJacobi}
\begin{algorithmic}[1]
\setlength{\itemsep}{3pt}
\Procedure {GLEvalJacobi}{$n, [t_0, t_1, \ldots, t_M], \p{W}$}

\State $\tau_1, \ldots, \tau_n \gets \texttt{roots\_legendre(n)}$
\For {$i \gets 1$ \textbf{to} $\lceil n/2 \rceil$}
  \State $P_1(\tau_i),P_2(\tau_i),\ldots,P_{n-1}(\tau_i) 
                         \gets\texttt{evaluate using \eqref{E:JacobiRecRel}}$
\EndFor
\For {$i \gets \lceil n/2 \rceil+1$ \textbf{to} $n$}
  \State $P_1(\tau_i),P_2(\tau_i),\ldots,P_{n-1}(\tau_i) 
                               \gets\texttt{use symmetry \eqref{E:P_k-symm}}$
\EndFor

\For {$k \gets 1$ \textbf{to} $n-1$}
  \State $g^{(0)}_{k-1} \gets -(2k+1)/(2k) \cdot P_{k}(\tau_1)$
  \For {$i \gets 1$ \textbf{to} $n-1$}
    \State $g^{(i)}_{k-1} \gets (2k+1)/(2k) 
                                      \cdot(P_{k}(\tau_i)-P_{k}(\tau_{i+1}))$
  \EndFor
  \State $g^{(n)}_{k-1} \gets (2k+1)/(2k) \cdot P_{k}(\tau_n)$
\EndFor

\For {$k \gets 0$ \textbf{to} $n-2$}
  \State $\p{S}_k \gets \sum_{i=0}^n g_k^{(i)} \p{W}_i$
\EndFor

\State $\p{V}_0 \gets  (\p{W}_0 + \p{W}_n)/2$
\State $\p{V}_1 \gets  (\p{W}_n - \p{W}_0)/2$
\For {$k \gets 0$ \textbf{to} $M$}
 \State $\p{R}_k \gets \p{V}_0 + t_k \p{V}_1 + (t_k^2 - 1)/2 \cdot 
             \texttt{ClenshawJacobi(}[\p{S}_0,\p{S}_1,\ldots,\p{S}_{n-2}], 
                                                          t_k, 1,1\texttt{)}$
\EndFor

\State \Return $[\p{R}_0, \p{R}_1, \ldots, \p{R}_M]$


\EndProcedure
\end{algorithmic}
\end{algorithm}

\newpage

\begin{algorithm}[ht!]
\caption{Multipoint evaluation of a GL curve using 
                                 a Legendre basis}\label{A:MultiEvalLegendre}
\begin{algorithmic}[1]
\setlength{\itemsep}{3pt}
\Procedure {GLEvalLegendre}{$n, [t_0, t_1, \ldots, t_M], \p{W}$}

\State $\tau_1, \ldots, \tau_n \gets \texttt{roots\_legendre(n)}$
\For {$i \gets 1$ \textbf{to} $\lceil n/2 \rceil$}
  \State $P_1(\tau_i),P_2(\tau_i),\ldots,P_{n-1}(\tau_i) 
                         \gets\texttt{evaluate using \eqref{E:JacobiRecRel}}$
\EndFor
\For {$i \gets \lceil n/2 \rceil+1$ \textbf{to} $n$}
  \State $P_1(\tau_i),P_2(\tau_i),\ldots,P_{n-1}(\tau_i) 
                               \gets\texttt{use symmetry \eqref{E:P_k-symm}}$
\EndFor

\For {$i \gets 1$ \textbf{to} $n$}
  \For {$k \gets 0$ \textbf{to} $n-1$}
    \State $R_k(\tau_i) \gets P_{k+1}(\tau_i)-P_{k-1}(\tau_i)$
  \EndFor
\EndFor

\For {$k \gets 0$ \textbf{to} $n-1$}
  \State $g^{(0)}_k \gets R_{k}(\tau_1)/2$
  \For {$i \gets 1$ \textbf{to} $n-1$}
    \State $g^{(i)}_k \gets \left(R_{k}(\tau_{i+1})-R_{k}(\tau_i)\right)/2$
  \EndFor
  \State $g^{(n)}_k \gets -R_{k}(\tau_n)/2$
\EndFor

\State $g^{(0)}_n \gets -P_{n-1}(\tau_1)/2$
\For {$i \gets 1$ \textbf{to} $n-1$}
    \State $g^{(i)}_n \gets \left(P_{n-1}(\tau_i) - 
                                              P_{n-1}(\tau_{i+1})\right)/2$
  \EndFor
\State $g^{(n)}_n \gets P_{n-1}(\tau_n)/2$

\For {$k \gets 0$ \textbf{to} $n$}
  \State $\p{S}_k \gets \sum_{i=0}^n g_k^{(i)} \p{W}_i$
\EndFor

\For {$k \gets 0$ \textbf{to} $M$}
 \State $\p{R}_k \gets (\p{W}_0 + \p{W}_n)/2 + 
              \texttt{ClenshawJacobi(}[\p{S}_0,\p{S}_1,\ldots,\p{S}_n], 
                                                       t_k, 0, 0\texttt{)}$
\EndFor

\State \Return $[\p{R}_0, \p{R}_1, \ldots, \p{R}_M]$


\EndProcedure
\end{algorithmic}
\end{algorithm}

\newpage
\bibliographystyle{elsart-num-sort}
\bibliography{GL-evaluation-BIT-rev}

\end{document}